\RequirePackage{etex}
\documentclass[10pt, twoside]{article}

\newcommand{\R}{\mathbb{R}}

\newcommand{\N}{\mathbb{N}}

\newcommand{\eps}{\varepsilon}

\newcommand{\ssup}[1]{{\scriptscriptstyle(#1)}}

\newcommand{\B}[1]{\boldsymbol{#1}}

\newcommand{\mc}{\mathcal}

\newcommand{\type}{\mathrm{\mathbf{type}}}
\newcommand{\lab}{\mathrm{\mathbf{label}}}

\newcommand{\EVn}[1]{\smash{\mathbb{E}_{\B n, \B {\kappa}^\ssup{n}}}\left[{#1}\right]}
\newcommand{\ld}[1]{\frac{1}{n}\log{#1}}

\newcommand{\Scal}{\mathcal{S}}
\newcommand{\Gcal}{\mathcal{G}}

\newcommand{\symm}{\smash{\mathrm{Sym}_d}}

\newcommand{\Obig}[1]{O\left({#1}\right)}
\newcommand{\floor}[1]{\lfloor{#1}\rfloor}

\newcommand{\e}{\mathrm{e}}
\usepackage[utf8]{inputenc}	
\usepackage{amsmath, amsthm, amsfonts, amssymb, amscd}
\usepackage{graphicx} 
\usepackage[english]{babel}
\usepackage{multirow,booktabs}
\usepackage[table]{xcolor}
\usepackage{fullpage}
\usepackage{tikz}
\usepackage{tikz-network}
\usepackage{pgfplots}
\usepackage{float}
\pgfplotsset{compat = newest}
\usetikzlibrary{arrows.meta}
\usetikzlibrary{bending}
\usetikzlibrary{decorations.markings}
\usetikzlibrary{shapes}
\usepackage{lastpage}
\usepackage{caption}
\usepackage[shortlabels]{enumitem}
\usepackage{comment}
\usepackage{fancyhdr}
\usepackage{mathrsfs}
\usepackage{wrapfig}
\usepackage{dsfont}
\usepackage{esint}
\usepackage{setspace}
\usepackage{commath}
\usepackage{calc}
\usepackage{multicol}
\usepackage{cancel}
\usepackage{systeme}
\usepackage{physics}
\usepackage[retainorgcmds]{IEEEtrantools}
\usepackage[margin=3cm]{geometry}
\usepackage{hyperref}
\hypersetup{
    colorlinks=true,
    linkcolor=blue, 
    urlcolor=red, 
    pdftitle = {How to write a thesis}
}
\usepackage[normalem]{ulem}               
\newcommand\redout{\bgroup\markoverwith
{\textcolor{red}{\rule[0.5ex]{2pt}{0.8pt}}}\ULon}

\usepackage{empheq}
\usepackage{framed}
\usepackage[most]{tcolorbox}
\usepackage{xcolor}
\colorlet{shadecolor}{orange!15}
\numberwithin{equation}{section}

\newtheoremstyle{def}
  {\topsep}{\topsep}%
  {\slshape}{}%
  {\bfseries}{}%
  {\newline}{}%

\newtheoremstyle{thm}
  {\topsep}{\topsep}%
  {\itshape}{}%
  {\bfseries}{}%
  {\newline}{}%

\theoremstyle{thm}
\newtheorem{theorem}{Theorem}[section]
\newtheorem{Lemma}[theorem]{Lemma}
\newtheorem{proposition}[theorem]{Proposition}
\newtheorem{corollary}[theorem]{Corollary}
\newtheorem{ass}[theorem]{Assumption}

\theoremstyle{def}
\newtheorem{definition}[theorem]{Definition}
\newtheorem{remark}[theorem]{Remark}

\begin{document}
\setcounter{section}{0}

\title
{Counting connected multitype graphs with prescribed edge counts}

 \author{Luisa Andreis\footnote{Department of Mathematics ``G. Peano'', Universit\`a degli Studi di Torino, Torino, Italy.\\
 Email: luisa.andreis@unito.it}, Mario Veshaj\footnote{Weierstrass Institute for Applied Analysis and Stochastics, Anton-Wilhelm-Amo-Strasse 39, 10117 Berlin, Germany.\\
 Email: veshaj@wias-berlin.de}}
\date{\today}  

 \maketitle

\abstract{
We determine the exponential asymptotics of the number of connected labelled graphs with finitely many vertex types, prescribed numbers of vertices of each type, and prescribed edge counts between each pair of types. Under an irreducibility and strict surplus condition, we express the count for linear surplus of edges as the unrestricted count multiplied by an explicit exponential connectivity factor, up to a multiplicative error \(\exp\{o(n)\}\). The factor is characterized by a finite-dimensional nonlinear system whose solution exists and is unique. We use a probabilistic  method, exploiting large-deviation principles and asymptotic estimates for connectedness probabilities to relate the counting problem to the emergence of giant components in suitably tuned supercritical random graphs. For connected bipartite graphs, we determine the exponential rate throughout the regime of positive linear excess and fixed positive limiting proportions of the two vertex classes. In the parameter ranges where the existing upper and lower bounds apply, this identifies the exponential rate left undetermined by those estimates.
 }




\section{Introduction}
The enumeration of connected graphs is a classical topic in combinatorics, dating back to Cayley’s celebrated formula, which states that the number of labelled spanning trees on $n$ vertices equals $n^{n-2}$. A major advance beyond trees was achieved by Bender, Canfield, and McKay~\cite{BeCaMcKa90}, who derived an asymptotic formula for the number of connected graphs with $n$ vertices and $\ell$ edges, in the regime where both $n$ and $\ell$ tend to infinity. In this context, the quantity $\ell - n + 1$ is commonly referred to as the \emph{excess} (or \emph{complexity}) of the graph, as it measures the deviation from the minimal number of edges required for connectivity.

The approach of~\cite{BeCaMcKa90} relies primarily on analytic methods based on differential equations, and several alternative proofs and refinements have since been developed. Pittel and Wormald~\cite{PiWo05} proposed an ``inside-out'' combinatorial strategy. They first count connected graphs with minimum degree at least two (the $2$-core), and then count the ways trees can be attached to this core to recover all connected graphs. Van der Hofstad and Spencer~\cite{VdHS06}  analysed a breadth-first exploration process on an Erd\H{o}s--R\'enyi random graph with a suitably chosen edge probability, and related it to an occupancy problem in which balls are placed into tilted bins. Coja-Oghlan, Moore and Sanwalani~\cite{CoMoSa07} took a different, particularly elegant and powerful, route. To count connected graphs with a given numbers of vertices and edges, the authors consider a larger random graph whose unique giant component has these precise features. This yields a short proof of the asymptotic formula of~\cite{BeCaMcKa90}, and it also allows to extend the result to hypergraphs. The argument is an instance of the probabilistic method~\cite{AlSp16}, in which properties of random structures are used to prove deterministic enumerative results and it was further explored to get other results in the hypergraph setting~\cite{BeCoKa14, BoRi16}. More recently, de~Panafieu~\cite{Pan19} used tools from analytic combinatorics to sharpen the asymptotic estimates of~\cite{BeCaMcKa90} in the sparse regime, where the number of edges grows linearly with the number of vertices.

A natural extension of the classical graph setting is obtained by considering \emph{multitype} (or \emph{colored}) graphs, in which each vertex is assigned a type from a set of types. Enumeration problems for multitype graphs are way less studied, see for example~\cite{dPRa15, dP15}, and among them, enumeration under connectivity constraint even less.
Indeed, while colored trees have been addressed, for instance in~\cite{BeMo14}, the combinatorial literature on multitype connected graphs remains relatively limited. However multitype graphs, beside being a natural generalization of the classical setting, allow to treat special graphs of interest. Take as an example the class of bipartite graphs, which can be described as a graph with vertices of two different colors and only cross-color connections allowed. Asymptotic enumeration of connected bipartite graphs is an open question and it has been recently addressed in the regime of a finite excess of edges, see~\cite{DoErKaMi23, Cl24, HaShYa25}.
By contrast, the probabilistic study of multitype graphs, often under the name of \emph{inhomogeneous random graphs}, is well developed. Starting from the seminal work of Bollob\'as, Janson, and Riordan~\cite{BJR07}, which established explicit criteria for the phase transition in multitype Erd\H{o}s--R\'enyi graphs (introduced earlier in~\cite{Sod02}), a substantial body of research has explored structural, spectral, and large-deviation properties of these models; see, for instance,~\cite{BhvdHovLe12, DeFr14, ChHaDHoSf21, AKLP23,  BhBuSa24, YuSu24}.

In the present work, we study the number of connected labelled graphs whose vertices have finitely many types, with prescribed numbers of vertices of each type and prescribed numbers of edges between each pair of types. In the sparse regime with linear excess, we determine this count up to a multiplicative factor \(\exp\{o(n)\}\), under an irreducibility and strict surplus condition. The result identifies the exponential cost of connectedness in the corresponding uniform graph ensemble and yields, in particular, an enumeration formula for connected bipartite graphs with linear excess.
Our approach belongs to the class of probabilistic methods and it is in line with the one developed in~\cite{CoMoSa07} for connected hypegraphs. We exploit the following fact: in the sparse regime, any connected multitype graph with given vertex and edge counts can be realized as the giant component of a suitably chosen supercritical inhomogeneous random graph. By carefully tuning the type distribution and the connection matrix, one can get close to the desired combinatorial constraints. Later, a continuity argument allow to meet the constraints precisely. Using large-deviation estimates for inhomogeneous random graphs from~\cite{AKLP23} we conclude. A related application of this larger-graph perspective appears in the recent work of~\cite{BeDovdH26}, where the authors use a similar strategy to study the enumeration of connected graphs with a prescribed degree sequence; both works were motivated in part by earlier joint discussions on connected random graphs.

The described procedure allows us to determine the exponential rate of the number of connected labelled multitype graphs with prescribed vertex-type counts and prescribed edge counts between each pair of types. Under an irreducibility and strict surplus condition, the count equals the number of unrestricted graphs with the same statistics multiplied by \(\exp\{n\varphi(\nu,a)+o(n)\}\). The rate \(\varphi\) is explicit in terms of the unique solution of a finite-dimensional nonlinear system. Equivalently, it gives the exponential decay rate of the probability of connectedness in the corresponding uniform graph ensemble. For bipartite graphs, the formula applies to arbitrary fixed positive limiting proportions of the two classes and every positive linear excess and it addresses a question raised in~\cite{DoErKaMi23, MiPeWa26}. 
While previous upper and lower bounds differed for a term $e^{\Theta(n)}$, we manage to get a difference of order $e^{o(n)}$.

\subsection{Connected bipartite graphs with linear excess}
Before stating our general result, let us present it in the special case of connected bipartite graphs, which we believe is of great interest on its own and it may be hard to grasp from the general formula. Following standard notation, see for instance~\cite{DoErKaMi23, Cl24}, let $\mathcal{C}(i,j,\ell )$ be the number of connected bipartite graphs with $i$ vertices in one partition class, $j$ in the second, and $i + j +\ell$ many edges. 

\begin{corollary}\label{cor:bipartite}
Let $\{(i_n,j_n,\ell_n)\}_{n\in\N}$ be a sequence of triplets in $\N^3$ such that $n=i_n+j_n$ and the following limits hold 
\[
\lim_{n\to \infty} \frac {i_n}{n}=p \in (0,1),\quad \lim_{n\to \infty} \frac {j_n}{n}=1-p \in (0,1),\quad  \lim_{n\to \infty} \frac {\ell_n}{n}=c \in (0,\infty).
\]
Then, when $n\to\infty$,
 \[
        \mathcal{C}(i_n,j_n,\ell_n ) = \binom{i_nj_n}{i_n+j_n+\ell_n} e^{n[\phi(p,c)+o(1)]} 
    \]
    where
    \[
    \phi(p,c)= p \log\left(\frac{2y_{1}}{1-y_{1}}\right)+(1-p) \log\left(\frac{2y_{2}}{1+y_{2}}\right) -(1+c)\left[1+\log \left(\frac{2y_1y_2}{y_1+y_2}\right) \right]
    \]
    and the vector $(y_1, y_{2})$ is the unique positive solution to the set of implicit equations
    \[
       \begin{dcases}
           (1+c)\frac{2y_{1}y_{2}}{y_{1}+y_{2}} &=  p\log\left(\frac{1+y_{1}}{1-y_{1}}\right),\\
           (1+c)\frac{2y_{1}y_{2}}{y_{1}+y_{2}}& =  (1-p)\log\left(\frac{1+y_{2}}{1-y_{2}}\right).
       \end{dcases} 
\]
\end{corollary}
Notice that, by Stirling expansion, the term $\binom{ij}{i+j+\ell}$ is proportional to $n^{n+\ell}e^{nC+o(n)}$, for a constant $C$ depending explicitly on $p$ and $c$. Corollary~\ref{cor:bipartite} is a consequence of our main result, Theorem~\ref{thm:main}. This result partially answer a question raised in~\cite{MiPeWa26}, where lower bounds for the quantity $C(i,j,\ell)$ are found, but they differ from the upper bounds in~\cite{DoErKaMi23} of $e^{\Theta(i+j)}$. In the case of linear excess of edges our result reduces the distance between the upper and the lower bound to $e^{o(i+j)}$.
\begin{remark}[The balanced case]\label{rem:balanced}
When $p=1/2$ the fix point equation for $y_1=y_2$ is exactly the same of~\cite{BeCaMcKa90}, while the $\phi$ takes the form
 \[
    \phi(p,c)=  \log\left(\frac{2y}{\sqrt{1-y^2}}\right) -(1+c)\log (ey)
    \]
    where 
$y$ solves
\[ 2(1+c)y =  \log\left(\frac{1+y}{1-y}\right)\]
Since one gets an additional $e^{-(n+\ell) \log 2}$ due to the asymptotics of $\binom{n^2/4}{n+\ell}$ divided by $\binom{n^2/2}{n+\ell}$, we immediately see that 
\begin{align*}
 \mathcal{C}\left(\frac n2,\frac n2,\ell  \right) &= 2^{-(n+\ell)}(ey)^{-(n+\ell)} \binom{\binom{n}{2}}{n+\ell} \left(\frac{2y}{\sqrt{1-y^2}}\right)^{n} e^{o(n)}\\
 &= 2^{-(n+\ell)+o(n)} \#\{\text{connected }G\text{ with }|V(G)|=n, |E(G)|=n+\ell\}.
\end{align*}
This equivalence has a
probabilistic meaning. Fix the bipartition
$\{1,\dots,n/2\}\cup\{n/2+1,\dots,n\}$ and let $G_{\mathrm{conn}}(n,n+\ell)$ be uniform among
connected graphs on $[n]$ with $n+\ell$ edges. Then the equality above reads 
\begin{align*}
2^{-(n+\ell)+o(n)}&=\frac{\mathcal{C}(\frac n2,\frac n2,\ell )}{\#\{\text{connected }G\text{ with }|V(G)|=n, |E(G)|=n+\ell\}}\\
&= \mathbb P\bigl(\text{every edge of } G_{\mathrm{conn}}(n,n+\ell) \text{ crosses the bipartition}\bigr).
\end{align*}
This means that, up to $e^{o(n)}$, it has the same probability as in the
\emph{unconditioned} uniform graph $G(n,n+\ell)$, where each edge crosses the
bipartition with probability $\frac{1}2+o(1)$. 

\end{remark}

\subsection{Connected multitype graphs with prescribed edge counts}    

After the statement of our result in the bipartite case, let us introduce some notation, necessary to state the result in full generality.

\begin{definition} [Multitype graph] \label{def:multi_graph}
    Fix a finite non empty set $\mc{S}$ with $d$ elements. Let $\B n = (n_r)_{r\in\mc{S}}$ be a $d$-tuple of nonnegative integers. A labelled multitype graph of \emph{profile} $\B n$ is a graph on the vertex set
    \begin{equation}\label{eq:vertex_set}
    V_{\B n} : = \{v=(r,i): r\in \mc{S}, i\in [n_r]\}.
    \end{equation}
    The vertex $v=(r,i)$ is said to have type $\type(v)=r\in\mc{S}$ and label $\lab(v)=i\in [n_r]$. The edge set of the graph is the set of unordered pairs of vertices, i.e. it is composed as follows:
    \[
    E_{\B n} := \{{\B e}=\{v,v'\}\colon v,v'\in V_{\B n}  \}.
    \]
    By abuse of terminology and notation, the edge $\mathbf{e}=\{v,v'\}$ is said to have type pair $\type({\B e})\colon =\{\type(v),\type(v')\}$.
\end{definition}
The size of the graph is $n = \smash{\sum_{r\in \mc S} n_r}$. For each multitype graph we can specify its \emph{edge matrix} $\B \eps $, that records the total number of edges between vertices of a given type pair. More precisely, given $r,s\in \mc{S}$,
\[
\eps_{rs} \colon = \#\{ \mathbf{e} \in  E_{\B n} \colon \type(\mathbf{e})=\{r,s\}\}.
\]
The total number of edges is $\eps = \smash{ \sum_{r\le s} \eps_{rs}} $. In the following we will consider only simple undirected multitype graphs, hence edge matrices are always symmetric. Moreover, let us define the matrix $\B J=(J_{rs})$, where $J_{rs}$ is the maximum number of edges for each pair of types allowed by the profile $\B n$, i.e.
\begin{equation}\label{eq:edges_max} 
J_{rs}=
    \begin{dcases}
        \binom{n_r}{2} &\qq{if} r= s,\\
         n_rn_s &\qq{if} r\ne s,
    \end{dcases}
\end{equation}
then we see that the matrix $\B \eps$ is dominated entrywise by the matrix $\B J$. Given two matrices $\B A=(A_{ij})$ and $\B B=(B_{ij})$, we use the notation ${\B A} \preceq {\B B}$ (resp. ${\B A} \prec {\B B}$), if $A_{ij}\leq B_{ij}$ for all $i,j$ (resp. if the inequality is strict for all $i,j$). In this context it is clear that the number of labelled multitype graphs with type profile $\B n$ and edge matrix $\B \eps$ is simply 
\begin{equation} \label{eq: nr_connected_graphs}
B(\B n, \B \eps)= \prod_{r\le s} \binom{J_{rs}}{\eps_{rs}}. 
\end{equation}
Finally, we are ready to write our quantity of interest
\[
C(\B n, \B \eps) \coloneq \text{\#\{labelled \textbf{connected} multitype graph with profile $\B n$ and edge matrix $\B \eps$\}}.
\]
We are interested in the asymptotics of such quantity for large $n$ in the sparse setting, that is when the number of edges is proportional to the number of vertices. We need nevertheless that the excess of edges (the difference between the number of edges and the minimal number of edges needed to connect the graph, those of a spanning tree) is proportional to $n$ too (in~\cite{VdHS06} this is considered the regime of a \emph{large} excess of edges).  For this reason, we phrase our assumptions in a more involved way.

\begin{ass}\label{ass:main}
 Let $\B \nu\succ \B 0$ be a probability vector on $\mc{S}$ and let $\B a$ be a symmetric $d\times d$ matrix with nonnegative entries. Suppose $\B a$ can be written as $2^{\delta_{rs}}a_{rs} = w_{rs}+w_{sr}$, where $\delta$ is the Kronecker delta and $\B w$  is a nonnegative and irreducible matrix satisfying $\B \nu \prec \B w \B 1$, with $\B 1$ the all ones vector.    
\end{ass}

Now we are ready to state our main theorem, which determines the exponential rate of the connected count $ C(\B n, \B \eps)$  relative to the unrestricted count $ B(\B n, \B \eps)$.
\begin{theorem}\label{thm:main}
   Let Assumption~\ref{ass:main} hold.    Let $\{\B n=(n_r)_{r\in \Scal}\}$ be a sequence of $d$-tuples with positive integer entries, such that $n=\sum_{r} n_r$ and  $\frac{1}{n}\B n\to \B \nu$ as $n\to\infty$. Let $(\B \eps_n)$ be a sequence of edge matrices such that $\frac 1n \B \eps_n\to \B a$.  The number of labelled connected multitype graphs $C(\B n, \B \eps_n)$ with profile $\B n$ and edge matrix $\B \eps_n$ edges reads
    \begin{equation}\label{eq:asympt_formula}
        C(\B n, \B \eps_n) = B(\B n, \B \eps_n) e^{n[\varphi(\B \nu,\B a)+o(1)]} 
    \end{equation}
    with 
    \begin{equation} \label{eq: varphi}
        \varphi(\B \nu,\B a) =  \sum_{r} \nu_r \log\left(\frac{2y_{r}}{1+y_{r}}\right) -\sum_{r\le s} a_{rs}\left[1+\log \left(\frac{2y_ry_s}{y_r+y_s}\right) -\left(\frac{2y_ry_s}{y_r+y_s}\right)\right]
    \end{equation}
    and the quantity $(y_{r})_{r\in \mc{S}}$ is the unique solution in $(0,1)^d$ for the set of implicit equations
    \begin{equation} \label{eq: y_rs_equation}
        \sum_{s=1}^d 2^{\delta_{rs}}a_{rs}\frac{2y_{r}y_{s}}{y_{r}+y_{s}} =  \nu_r\log\left(\frac{1+y_{r}}{1-y_{r}}\right),
\end{equation}
for all $r\in \mc{S}$.
\end{theorem}
Our result can be seen as a large deviation statement for the random graph $\mathcal{G}(\B n, \B \eps_n)$ uniformly chosen among the ones with type profile $\B n$ and edge matrix $\B \eps$. Then \eqref{eq:asympt_formula} gives the asymptotic exponential cost of the rare event that such graph is connected, i.e.
\[
\lim_{n\to \infty} \frac 1n \log \mathbb{P}(\mathcal{G}(\B n, \B \eps_n) \text{ is connected})=\varphi(\B \nu,\B a).
\]
In order to see \eqref{eq:asympt_formula} as the natural multitype generalization of formula (1.4) in~\cite{BeCaMcKa90}, notice that the product of binomial factors is simply the number of different multitype graphs that can be constructed with profile $\B n$ and edge matrix $\B \eps_n$. Then, writing $\sum_{r \le s} a_{rs}\frac{2y_{r}y_{s}}{y_{r}+y_{s}}=\frac{1}{2}\sum_{r,s} 2^{\delta_{rs}}a_{rs}\frac{2y_{r}y_{s}}{y_{r}+y_{s}} = \sum_r \nu_r\log\sqrt{\frac{1+y_{r}}{1-y_{r}}}$:
\begin{align*}
        \varphi({\B{\nu}},\B{a})  &=  \sum_r \log \left[\left(\frac{2y_{r}}{\sqrt{1-y^2_{r}}}\right)^{\nu_r}\left(ey_{r}\right)^{-a_{rr}} \right] + \sum_{r<s}\log \left(e{\frac{2y_ry_s}{y_r+y_s}}\right)^{-a_{rs}}.
\end{align*}
We immediately recognise a diagonal term taking exactly the same form of the main exponential term in~\cite{BeCaMcKa90}, while the off diagonal is a novelty that we could not guess from the single-type case.

Putting together the proof of Theorem~\ref{thm:main} and the proof of Theorem~3.6 from~\cite{AKLP23}, we see that a connected multitype graph with vertex profile around $ \B n$ and edge matrix around $ \B \eps_n$ can be realised with high probability as the (unique) giant component of the inhomogeneous random graph $\Gcal(\B n^*, \B \kappa^*)$ with vertex profile $\B n^*= \lfloor n \B \nu^*\rfloor$ and edge probability $\B p^*= \frac 1n \B\kappa^*$, where
\begin{align*}
\nu^*_r&=\frac{\nu_r}{1-\e^{-(\B \kappa^*\B \nu)_r}}\qquad r \in \mc{S};\\
 \kappa^*_{rs}&=   \frac{2^{\delta_{rs}}a_{rs}}{\nu_r\nu_s}\frac{2y_ry_s}{y_r+y_s} \qquad r,s \in \mc{S}.
\end{align*}
This strategy of building a larger random structure with a given giant component is the core argument of~\cite{CoMoSa07}, where it is employed to get asymptotics for the number of hypergraphs with given number of vertices and edges. 



Our formula identifies the exponential cost of connectedness under simultaneous constraints on all vertex-type and edge-type counts. Obtaining a relative asymptotic equivalent would require additional control of the subexponential factor. We believe that finer asymptotics for the giant component of inhomogeneous random graphs, like joint local central limits for the number of vertices and edges, are on a reach with the approach from~\cite{BeCoKa14}. This would be useful in understanding the constant prefactor ($\e^{a(x)}$ in~\cite{BeCaMcKa90}) in the multitype case as well.

\begin{remark}[On the assumptions for connectedness]
Notice that the assumptions on $\B \eps$ such that there exists at least one \emph{connected} multitype graph $(V_{\B n}, E_{\B n})$ with profile ${\B n}$ and edge matrix ${\B \eps}$ are non-trivial as soon as the number of types $d>1$. In the one type case, it is enough to have at least $n-1$ edges over $n$ vertices to connect all of them, while assumptions for $d>1$ involve the profile $\B n$ and the matrix $\B \eps$ non-trivially and, to the best of our knowledge, they are not know explicitly for general $d$, see Section~\ref{sec: trees}.
Indeed, condition $\B \nu \preceq \B w\B 1$ is actually necessary for $\B a$ to be the limit of the normalized edge matrix of a connected multitype graph (see section \ref{sec: trees}), but it is not sufficient. What we require in addition for the theorem to hold is that we can orient the edges of the graph (the orientation is prescribed by the matrix $\B w$) such that this orientation is irreducible and it has a linear surplus of outgoing edges from every type (i.e. $\nu_r<\sum_{s}w_{rs}$ for any $r$).  This allows us to prove that~\eqref{eq: y_rs_equation} has a solution away from the boundary (i.e. strictly positive in all entries). The existence of such a vector $(y_{r})_{r\in \mc{S}}$ is enough for the proof to be carried out, hence a \emph{sufficient} condition for $\B a$ to be the rescaled edge matrix of a sequence of \emph{connected} multitype graphs.
\end{remark}

\section{Proof of Theorem~\ref{thm:main}}

This section is divided into several subsections, each collecting auxiliary results needed for the proof of our main theorem. In the final subsection, we combine these ingredients to prove Theorem~\ref{thm:main}.
\subsection{Inhomogeneous random graphs and connectedness}
In order to study the quantity $  C(\B n, \B \eps_n)$, we introduce the inhomogeneous random graph model, that is a natural generalisation of the famous Erd\H{o}s-R\`enyi random graph, see the seminal paper~\cite{BJR07}. 

Let the non-empty finite set $\mc{S}$ be fixed. We consider a random graph on the vertex set $[n]=\{1,\dots, n\}$ and assign a (possibly random) type from $\mc{S}$ to each vertex. We write $n_r$ for the number of vertices with type $r\in \Scal$, we collect it in the type vector $\B n = (n_r)_{r\in \mc{S}}$, and we assume $\B n/n \to \B \nu \succ \B 0$ as $n\to \infty$. 
The edges of this graph are undirected and randomly drawn; self and multiple edges are excluded. The $\binom{n}{2}$ possible edges are sampled independently.
The probability to draw an edge between two vertices with types $r$ and $s$ is $p_{rs}$; this defines a map $\B p\colon\mc{S}\times \mc{S}\to[0,1]$. The resulting random graph $\Gcal$ is called the inhomogeneous random graph on type vector $\B n$ and function of probabilities $\B p$. In particular, we are interested in the sparse regime as $n\to\infty$, i.e. in the case where the expected number of edges per vertex is of finite order (or, equivalently, the expected number of edges in the graph is proportional to the size of the graph). This is the case if the probabilities $p_{rs}$ are of order $1/n$. Actually, we impose that they are given by
$$
p^\ssup{n}_{rs}=1\wedge \frac 1n \kappa^\ssup{n}_{rs},\qquad r,s\in\Scal,
$$
where $\B \kappa^\ssup{n}\colon \Scal\times\Scal\to[0,\infty)$ is a symmetric non-negative bounded function, called the \emph{connection kernel}. We focus on the case in which $\B \kappa^\ssup{n}\to \B \kappa$ as $n$ grows to infinity, for a limiting connection kernel $\B \kappa$ (a symmetric $d\times d$ matrix with non-negative entries in our case). We suppose that $\B \kappa$ is  irreducible. Without loss of generality, we are from now on assuming that $\frac 1n\|\B \kappa^\ssup{n}\|_{\infty}\leq 1$ (always true for $n$ large enough), and hence $\kappa^\ssup{n}_{rs}/n$ is a probability for any $r,s\in\Scal$. We write $\Gcal(\B n, \B \kappa^\ssup{n})$ to indicate the resulting sparse random graph and we use, respectively, $\smash{\mathbb{P}_{\B n, \B \kappa^\ssup{n}}}$ and $\smash{\mathbb{E}_{\B n, \B \kappa^\ssup{n}}}$ for its probability law and expectation. The probability of any given realization $G$ of the random graph with edge configuration $\B \eps$ is given by the term
\begin{equation} \label{eq: prob_realization_eps}
P_{\B n, \B \kappa^\ssup{n}}(\B \eps) = \prod_{r\le s}\left(\frac{\kappa_{rs}^\ssup{n}}{n}\right)^{\eps_{rs}}\left(1-\frac{\kappa_{rs}^\ssup{n}}{n}\right)^{J_{rs}-\eps_{rs}},
\end{equation}
with the convention $0^0\coloneqq 1$ and $\B J$ defined in \eqref{eq:edges_max}. 

Let $\symm$ be the space of symmetric $d\times d$ matrices equipped with the inner product $\langle \B A,\B B\rangle\coloneqq \sum_{i\leq j} A_{ij}B_{ij}$. We have that $\symm$ is dual to itself, hence $\symm^* = \symm$. Let  $\smash{{\mc E}_{rs}^\ssup{n}}$ be the number of edges between pairs of type $r$ and $s$ in $\Gcal(\B n, \B \kappa^\ssup{n})$ and consider the $\symm $ - valued random variable $\smash{\B {\mc E}^\ssup{n} = \smash{(\mc E}_{rs}^\ssup{n})_{r,s\in \Scal^2}}$. The probability that the inhomogeneous random graph has edge configuration $\B \eps\in \symm$ is 
\[
    \mathbb{P}_{ \B n, \B \kappa^\ssup{n}} \left({\B {\mc E}^\ssup{n}=\B \eps} \right) = B(\B n, \B \eps) P_{\B n, \B \kappa ^\ssup{n} }(\B \eps), 
\]
where $B(\B n, \B \eps)$ is the number of labelled multitype graphs with type profile $\B n$ and edge matrix $\B \eps$ defined in \eqref{eq: nr_connected_graphs}. Let $\B t \in \symm $. Then  the normalized logarithmic moment generating function of the random variable $\B {\mc{E}}^\ssup{n}$ has limit
\begin{align} \label{eq: multid_moment_gen_fun}
    M_{\B \nu, \B \kappa}(\B t) \coloneqq \lim_{n\to \infty}  \frac{1}{n}\log{\mathbb{E}_{\B n, \B \kappa^\ssup{n}} \left[ e^{\left\langle\B{t}, \B{\mc{E}}^\ssup{n} \right\rangle} \right]} =\frac 12 \sum_{r,s\in \Scal} {\nu_r \kappa_{rs} \nu_s} (e^{t_{rs}}-1).
\end{align}
We further introduce some notation for the probability of the graph to be connected and to be connected with given edge statistics:
\begin{equation}\label{eq:prob_connectedness}
    q_n(\B n, \B \kappa^\ssup{n}) \coloneq \mathbb{P}_{ \B n, \B \kappa^\ssup{n}} \left(\Gcal\text{ is connected}\right),\quad
    q_n(\B n, \B \kappa^\ssup{n}, \B \eps)  \coloneq \mathbb{P}_{ \B n, \B \kappa^\ssup{n}} \left(\Gcal\text{ is connected}, \B {\mc E}^\ssup{n}=\B \eps\right).
\end{equation}
In order to prove our result, we crucially use  the logarithmic asymptotics of the probability of connectedness, $q_n(\B n, \B \kappa^\ssup{n})$ from~\cite[Theorem 3.6]{AKLP23}, which we state here for completeness.
\begin{theorem} [Limit of probability of connectedness from~\cite{AKLP23}]\label{th: limit_connectivity}
Let $\B \nu$ be a probability vector on $\Scal$ and $\B \kappa$ an irreducible kernel. Let $\B n/n\xrightarrow{} \B \nu$ and $\B \kappa^\ssup{n} \xrightarrow{} \B \kappa$, then
    \begin{equation}\label{eq:connect_LDP}
        \lim_{n\to \infty} \ld{q_n(\B n, \B \kappa^\ssup{n})} = \left\langle \B \nu, \log\left(1-e^{-\B{\kappa}\B{\nu}}\right)\right\rangle \eqcolon{q}(\B \nu, \B \kappa).
    \end{equation}
\end{theorem}
Loosely speaking, the idea of the proof in~\cite[Theorem 3.6]{AKLP23} consists in finding a \emph{larger inhomogeneous graph} with the same connection kernel $\B \kappa^\ssup{n}$ whose (unique) giant component has profile $\approx \B n$.  Here, we use the fact that 
\begin{equation} \label{eq: explicit_q}
    q_n(\B n, \B \kappa^\ssup{n}) = \sum_{\B \eps} q_n(\B n, \B \kappa^\ssup{n}, \B \eps) \qquad q_n(\B n, \B \kappa^\ssup{n}, \B \eps) = C(\B n, \B \eps )P_{\B n, \B \kappa ^\ssup{n}}(\B \eps),
\end{equation}
where the sum runs over all possible edge matrices $\B \eps$ allowed by the profile $\B n$. Using a large deviation principle for the edge matrix of the graph conditioned on being connected, we will prove that such a sum rapidly concentrates (at an exponential speed) around one typical value $\B \eps^*$, depending on $\B \nu$ and $\B \kappa$. Hence, tuning carefully the connection kernel $\B \kappa$ one can find an inhomogeneous graph that has the desired vertices profile (i.e. $\approx \B n$) and the desired edge matrix ($\approx \B \eps$).
\subsection{Edge counts conditional on connectedness}

 We give here a large deviation result for the edge matrix of an inhomogeneous random graph conditioned to be connected\footnote{The proof of this LDP comes from an ongoing joint work on connected random graphs of the first author together with B. Chin, M. Dickson, S. Donderwinkel, R. van der Hofstad and N. Malhotra. The work is still in progress and we intend to be self-contained, we write here the proof of this proposition as well in the special finite-type case.}. Fix a probability vector $\B \nu$ on $\mc{S}$ and let us define, for an irreducible connection kernel $\B \kappa$ and $\B t\in \symm$, the matrix $\tilde{\B{\kappa}}({\B{t}})=(\e^{t_{rs}}\kappa_{rs})_{r,s\in \mc{S}} \in \symm$. We define
\begin{equation}
    \Lambda_{\B \nu, \B \kappa} (\B{t}) \coloneq  M_{\B \nu, \B \kappa}(\B t)    +   q(\B \nu, \tilde{\B{\kappa}}({\B{t}})) -  q(\B \nu, \B{\kappa}) = \left \langle \B \nu, \log\left[\frac{\sinh \left(\frac{1}{2} \tilde{\B{\kappa}}(\B{t})\B{\nu} \right)}{\sinh \left(\frac{1}{2} \B{\kappa}\B{\nu} \right)}\right]\right \rangle,
\end{equation}
by combining \eqref{eq: multid_moment_gen_fun} and \eqref{eq:connect_LDP}. Consequently, the Legendre transform of $ \Lambda_{\B \nu, \B \kappa} $ writes as
\begin{equation}\label{eq:rate_funct}
    \mc{J}_{\B \nu, \B \kappa}(\B A)=\sup_{\B t \in \symm }\{ \langle \B t, \B A\rangle - \Lambda_{\B \nu, \B \kappa} (\B{t})\},
\end{equation}
for any $\B A\in \symm$. 
\begin{proposition}[LDP for the number of edges given connectedness] \label{prop: LDP}
        Let $\B{\mc{E}}^\ssup{n}$ be the random edge configuration in the inhomogeneous random graph $\Gcal(\B n, \B \kappa^\ssup{n})$ with $\B n/n\xrightarrow{} \B \nu$ and $\B \kappa^\ssup{n} \xrightarrow{} \B \kappa$. Then, the sequence of laws 
    $ \mathbb{P}_{ \B n, \B \kappa^\ssup{n}} \left(\frac{\B{\mc{E}}^\ssup{n}}{n} \in \cdot \ \vert \ \Gcal(\B n, \B \kappa^\ssup{n}) \text{ is connected} \right)$
    satisfies an LDP with speed $n$ and rate function $\mc{J}_{\B \nu, \B \kappa}$ from~\eqref{eq:rate_funct}. Moreover, the rate function has unique global minimizer $\B{a}^*\in \symm$ given by
    \begin{align} \label{eq: typical_edges}
    a^*_{rs} = \frac{1}{2^{\delta_{rs}}} \frac{\nu_r \kappa_{rs}\nu_s}{2} \left(\coth\left(\frac{1}{2}(\B{\kappa}\B{\nu})_r \right)  +\coth\left(\frac{1}{2}(\B{\kappa}\B{\nu})_s\right)\right), \quad r,s\in \Scal.
\end{align}\end{proposition} 

\begin{proof}
We will rely on G\"artner-Ellis theorem, see~\cite[Theorem~2.3.6]{DZ10}. Namely, first we derive the logarithmic asymptotics of the moment generating function and then we deduce the large deviation principle. Let $\B{t} \in \symm $ and consider the moment generating function of $\smash{\B{\mc{E}}^\ssup{n}}$ given the connectedness of the random graph $\Gcal = \Gcal(\B n, \B \kappa^\ssup{n})$. We can write
\begin{align*}
    {\EVn{\e^{\left\langle\B{t}, \B{\mc{E}}^\ssup{n}\right\rangle} \ \vert \ \Gcal \text{ connected} }} & = \frac{\EVn{\e^{\left\langle\B{t}, \B{\mc{E}}^\ssup{n}\right\rangle} \mathds{1}_{\Gcal \text{ connected}} }}{q_n(\B n, \B \kappa^\ssup{n})}.
\end{align*}
Let $\B{\eps}(G)$ be the edge matrix of a connected multitype graph $G$ of profile $\B n$ and recall \eqref{eq: prob_realization_eps}. We can write
\begin{align}
   \EVn{ {\e^{\left\langle\B{t},  \B{\mc{E}}^\ssup{n}\right\rangle} \mathds{1}_{\Gcal \text{ connected}} }}  &= \sum_{G\text{ connected}} P_{\B n, \B \kappa ^\ssup{n}}(\B \eps(G)) \prod_{r\le s}\e^{t_{rs} \eps_{rs}(G) } \notag \\ 
   & = {\EVn{e^{\left\langle\B{t}, \B{\mc{E}}^\ssup{n}\right\rangle}}} \sum_{G\text{ connected}} P_{\B n, {\B \kappa} ^\ssup{n}}(\B \eps(G)) \prod_{r\le s} \frac{e^{t_{rs}\eps_{rs}(G)}}{(1+p^\ssup{n}_{rs}(\e^{t_{rs}}-1))^{J_{rs}}} \notag \\
   & = {\EVn{\e^{\left\langle\B{t}, \B{\mc{E}}^\ssup{n}\right\rangle}}} {q_n(\B n, \tilde{\B{\kappa}}_n({\B{t}}))}, \notag
\end{align}
where the kernel sequence $\tilde{\B{\kappa}}^\ssup{n}({\B{t}})$ is defined as 
\begin{align*}
    (\tilde{\B{\kappa}}^\ssup{n}({\B{t}}))_{rs} &=  \frac{ e^{t_{rs}}  \kappa^\ssup{n}_{rs} }{1+ p^\ssup{n}_{rs}\left(e^{t_{rs}}-1\right)} =  e^{t_{rs}}{\kappa^\ssup{n}_{rs}}+\Obig{\frac{1}{n}}
    \end{align*}
for all $r,s\in \mc{S}$, and $\tilde{\B{\kappa}}^\ssup{n}({\B{t}}) \xrightarrow{n\to\infty}  (e^{t_{rs}}\kappa_{rs})_{r,s\in \Scal} \eqcolon {\tilde{\B{\kappa}}(\B{t}})$.
Using all the above, combined with equation \eqref{eq: multid_moment_gen_fun} and Theorem~\ref{th: limit_connectivity}, we find that
 \begin{align*}
    {\EVn{\e^{\left\langle\B{t}, \B{\mc{E}}^\ssup{n}\right\rangle} \ \vert \ \Gcal \text{ connected} }} & = \EVn{\e^{\left\langle\B{t}, \B{\mc{E}}^\ssup{n}\right\rangle}  }\frac{q_n(\B n, \tilde{\B{\kappa}}^\ssup{n}({\B{t}}))}{q_n(\B n, \B \kappa^\ssup{n})},
\end{align*} hence
$\lim_{n\to \infty}\frac{1}{n} \log {\mathbb{E}_{\B n, \B \kappa^\ssup{n}}[{\e^{\langle\B{t}, \B{\mc{E}}^\ssup{n}\rangle} \ \vert \ \Gcal  \text{ connected} }]}= \Lambda_{\B \nu, \B \kappa} (\B{t}).$
By the G{\"a}rtner-Ellis theorem, the law of the random variable $\smash{\B{\mc{E}}^\ssup{n}/n}$ conditioned on the connectedness of the graph satisfies a large deviation principle with rate function $ \mc{J}_{\B \nu, \B \kappa}$.
The function $x\mapsto \log (\frac{1}{2}\sinh e^x)$ is convex on $\R$ and differentiable in zero, therefore $\Lambda_{\B \nu, \B \kappa} $ is convex as well 
and consequently the unique global minimizer of $\mc{J}_{\B \nu, \B \kappa}$ is achieved at $\B{a}^*=\grad{\Lambda}_{\B \nu, \B \kappa}(\B{0})$, that is given by \eqref{eq: typical_edges}.
\end{proof}
Note that $\B a^*$ can be written as $2^{\delta_{rs}}a_{rs}=w_{rs}+w_{sr}$, where 
$w_{rs} = \frac{\nu_r \kappa_{rs} \nu_s}{2} \coth \left(\frac{1}{2}(\B \kappa \B \nu)_r\right)$. The formulae of both $\B a^*$ and $\B w$ are always well-defined, as a consequence of the irreducibility of the connection kernel $\B \kappa$. Moreover, $w_{rs}=0  \iff a^*_{rs}=0 \iff \kappa_{rs}=0$, thus both $\B a^*$ and $\B w$ are irreducible as well. One can easily check that $\B w \B 1 \succ \B \nu$, hence, by the criterion \eqref{eq:directed-cofactor-revised} in the appendix, for all $\B A \succeq \floor{n(1-\delta)\B a^*}$ for $\delta>0$ small enough
\begin{equation} \label{eq: existence_optimal_graph}
    C(\B n, \B A)>0.
\end{equation}
Note finally that $\sum_{r\le s}a^*_{rs}>1$, which means that, asymptotically, the optimal connected graph is far from being a tree.

\subsection{Getting the prescribed edge counts}
Next proposition is a continuity result. The proof uses the large deviation result for the edge matrix in the connected graph to say that the probability of the graph to be connected is exponentially equivalent to the probability that the graph is connected with a given edge matrix. 

\begin{proposition}
    \label{th: typical_size_multid}
  Consider the inhomogeneous random graph $\mathcal{G}(\B{n},{\B{\kappa}})$ with $\B n/n\xrightarrow{} \B \nu$ and $\B \kappa$ irreducible kernel. Let $\B{a}^*$ as defined in \eqref{eq: typical_edges} and recall the probabilities defined in~\eqref{eq:prob_connectedness}. Let $(\B \eps_n)$ be a sequence of edge matrices such that $\frac 1n \B \eps_n\to \B a^*$. Then, there exists $\B \kappa^\ssup{n}\to \B \kappa$ such that
    \[
    \lim_{n\to \infty} \ld {q_n(\B n, \B \kappa) } = \lim_{n\to \infty} \ld{ q_n(\B n, \B \kappa^\ssup{n}, \B \eps_n)}.
    \]
\end{proposition} 
\begin{proof} 
In one sense it is trivial, since for all $\B \kappa^\ssup{n}\to \B \kappa$ we have that $\lim_{n\to \infty} \ld{q_n(\B n, \B \kappa^\ssup{n})} = \lim_{n\to \infty} \ld{q_n(\B n, \B \kappa)}$ by Theorem~\ref{th: limit_connectivity}, therefore
\[
\limsup_{n\to \infty} \ld{ q_n(\B n, \B \kappa^\ssup{n}, \B \eps_n)}\leq \lim_{n\to \infty} \ld{q_n(\B n, \B \kappa^\ssup{n})} =  \lim_{n\to \infty} \ld {q_n(\B n, \B \kappa) }.
\]

In order to lower bound this quantity, we will mainly use the following comparison between combinatorial terms $C(\B n,\B \eps)$ for different $\B \eps\in \symm$. Fix the pair $(r,s)$ and consider a connected multitype graph with $\eps_{rs} - 1$ edges of type $rs$, that is a connected graph with edge matrix $\B \eps - \B e_{rs}$, with $\B e_{rs}$ element of the canonical basis of $\symm$. Choose any of its $J_{rs}- (\eps_{rs}-1)$ absent edges and distinguish it, obtaining $(J_{rs}+1-\eps_{rs})C(\B n,\B \eps-\B e_{rs})$ different (distinguished) connected graphs. Whichever of these can be obtained by selecting a proper connected graph with edge matrix $\B \eps$ and distinguishing one of its $\eps_{rs}$ edges. Thus, $\eps_{rs} C(\B n, \B \eps) \ge (J_{rs} - (\eps_{rs}-1))C(\B n, \B \eps - \B e_{rs})$. Consequently,
\begin{equation}\label{eq:lower_bound_C}
    C(\B n, \B \eps) \ge  C(\B n, \B \eps-\B h) \prod_{r\le s}\left(\frac{J_{rs}-\eps_{rs}}{\eps_{rs}}\right)^{h_{rs}}
\end{equation}
for any $\B 0 \preceq \B h \preceq \B \eps$. Using the second of \eqref{eq: explicit_q}, we obtain
\[
 q_n(\B n, \B \kappa^\ssup{n}, \B \eps) \ge  q_n(\B n, \B \kappa^\ssup{n}, \B \eps -\B h )\exp\left( \sum_{r\le s}h_{rs} \log \left(\frac{(J_{rs}-\eps_{rs})}{\eps_{rs}}\frac{p^\ssup{n}_{rs}}{(1-p^\ssup{n}_{rs})}\right)\right).
\]
Now consider the sequence of edge matrices $\B \eps_n$. Let $S^* =\{(r,s)\in \Scal^2: \lim_{n\to \infty} \eps_{rs}/n = 0\}$, that is the set of index pairs for which $a^*_{rs}=0$. By equation \eqref{eq: typical_edges}, $S^*$ corresponds precisely to the set of index pairs for which $\kappa_{rs} = 0$, that is $S^* =\{(r,s)\in \Scal^2: \lim_{n\to \infty} \kappa^\ssup{n}_{rs} = 0\}$. Let $\B h_n$ be such that $h_{rs} = \eps_{rs}$ if $(r,s)\in S^*$, $h_{rs} = 0$ otherwise, and define $\hat{\B \eps}_n \coloneqq {\B \eps}_{n} - \B h_n$, whose entries are either $0$ or larger than $\delta n$ for some $\delta>0$. As long as $\lim_{n\to \infty}\frac{h_{rs}}{n}\log \kappa_{rs}^\ssup{n}=0$, we have 
\[
q_n(\B n, \B \kappa^\ssup{n}, \B \eps_n)\geq q_n(\B n, \B \kappa^\ssup{n}, \hat{\B \eps}_n)e^{o(n)}.
\]
Next, we choose $\kappa^\ssup{n}_{rs} = \kappa_{rs}$ for all $(r,s)\notin S^*$. This provides us a further lower bound, as
\[
q_n(\B n, \B \kappa^\ssup{n}, \hat{\B \eps}_n)=q_n(\B n, \B \kappa, \hat{\B \eps}_n)\prod_{(r,s)\in S^*, r\le s} \left(1-\frac{\kappa_{rs}^{(n)}}{n}\right)^{J_{rs}}=q_n(\B n, \B \kappa, \hat{\B \eps}_n)e^{o(n)},
\]
and we conclude that $q_n(\B n, \B \kappa^\ssup{n}, \B \eps_n) \ge q_n(\B n, \B \kappa, \hat{\B \eps}_n)e^{o(n)}$.

As a final step, we slightly decrease the connection kernel $\B \kappa=(1-\delta)\B \kappa$. Let $\B a^*(\delta)$ be defined in~\eqref{eq: typical_edges} with $\B \kappa'$ instead of $\B \kappa$. Since for all $r\in \mathcal{S}$ we have that $x\to x\coth(x (\B \kappa \B \nu)_r)$ is increasing for $x>0$, we see from~\eqref{eq: typical_edges} that $\B a^*(\delta)\prec \B a^*$ for $\delta\in (0,1)$. Now, for $\eta>0$, let us introduce the following set
\[\mc{R}_{\delta,\eta}(\B \kappa)=\left\{\B b=(b_{rs})\in \symm \colon \sup_{rs: \kappa_{rs}>0}|a_{rs}^*(\delta)-b_{rs}|\leq \eta, b_{rs}=0 \text{ if }\kappa_{rs}=0\right\}.\] For $\eta$ small enough (with respect to $\delta$), we have that $\hat{\B \eps}_n\succeq \B B$ whenever $\B B/n\in \mc{R}_{\delta,\eta}(\B \kappa)$. Let $\hat{q}_n (\mc{R}_{\delta,\eta}(\B \kappa)) = \sum_{\B B/n \in\mc{R}_{\delta,\eta}(\B \kappa)} q_n(\B n, \B \kappa', \B B)$. We can write
\begin{align*}
\hat{q}_n \left(\mc{R}_{\delta,\eta}(\B \kappa)\right) 
&= q_n(\B n, \B \kappa', \hat{\B \eps}_n) \left[\sum_{\B B/n \in \mc{R}_{\delta,\eta}(\B \kappa)} \frac{C(\B n, \B B)}{C(\B n, \hat{\B \eps}_n)}\prod_{r\le s}\left(\frac{\kappa'_{rs}/n}{1-\kappa'_{rs}/n}\right)^{B_{rs}-\hat{\eps}_{rs}}\right],
\end{align*}
where we recall that, $\hat{\B \eps}_n\succeq \B B$ and $C(\B n, \hat{\B \eps}_n)\ne 0$ by \eqref{eq: existence_optimal_graph}. Using again~\eqref{eq:lower_bound_C}, we see that 
\[
\frac{C(\B n, \B B)}{C(\B n, \hat{\B \eps}_n)}\leq \prod_{r\leq s}\left( \frac{J_{rs}-\hat{\eps}_{rs}}{\hat{\eps}_{rs}}\right)^{B_{rs}-\hat{\eps}_{rs}}.
\]
Now, observe that $\frac{\hat{\eps}_{rs}(1-\kappa'_{rs}/n)}{{(J_{rs}-\hat{\eps}_{rs})\kappa'_{rs}/n}} =O(1)$, so, for proper constants $1<C_{rs}<\infty$, we have
\begin{align*}
\frac{\hat{q}_n \left(\mc{R}_{\delta,\eta}(\B \kappa)\right) }{q_n(\B n, \B \kappa', \hat{\B \eps}_n) } &\le \sum_{\B B/n \in \mc{R}_{\delta,\eta}(\B \kappa)} \prod_{r\le s}\left(\frac{\hat{\eps}_{rs}(1-\kappa'_{rs}/n)}{{(J_{rs}-\hat{\eps}_{rs}) \kappa'_{rs}/n}}\right)^{\hat{\eps}_{rs}-B_{rs}} \\
&\le \sum_{\B B/n \in \mc{R}_{\delta,\eta}(\B \kappa)} \prod_{r\le s} C_{rs}^{\hat{\eps}_{rs}-B_{rs}} \le (2\delta n)^{d(d+1)/2}\e^{n2\delta c},
\end{align*}
for some finite constant $c<\infty$. Hence, we get
\[
 q_n(\B n, \B \kappa, \hat{\B \eps}_n)\geq  q_n(\B n, \B \kappa', \hat{\B \eps}_n)e^{O(n\delta)+o(n)}\geq \hat{q}_n \left(\mc{R}_{\delta,\eta}(\B \kappa)\right)e^{O(n\delta)+o(n)},
\]
where the first inequality is due to the comparison for a fixed edge count $\hat{\B \eps}_n$ between $\B\kappa$ and $(1-\delta)\B \kappa$, that is
\[
\frac{q_n(\B n, \B \kappa, \hat{\B \eps}_n)}{q_n(\B n, \B \kappa', \hat{\B \eps}_n)}=(1-\delta)^{-\sum_{r\le s} \hat{\eps}_{rs}}\prod_{r\leq s}\left(\frac{1-\frac{\kappa_{rs}}{n}}{1-(1-\delta)\frac{\kappa_{rs}}{n}}\right)^{J_{rs}-\hat{\eps}_{rs}}=e^{O(\delta n)+o(n)}.
\]
On the other hand, as a consequence of the large deviation principle for the edge matrix given connectedness (Proposition~\ref{prop: LDP}), we have that the probability of the event \emph{$\{G$ is connected with edge matrix close to $n\B a^*(\delta)\}$} is exponentially close to the probability of the event \emph{$\{G$ is connected$\}$}, hence
\begin{align*}
 \lim_{\eta\searrow 0} \liminf_{n\to \infty} \ld {\hat{q}_n \left(\mc{R}_{\delta,\eta}(\B \kappa)\right)}\geq  \lim_{n\to \infty} \ld {q_n(\B n, \B \kappa') } +O(\delta).
\end{align*}
Finally, as $\delta\searrow 0$ the right-hand side converges to $\lim_{n\to \infty} \ld {q_n(\B n, \B \kappa) }$, proving the thesis.
\end{proof}

\subsection{Calibration of the connection kernel}
Finally we need to ensure the existence of solutions of the connected fixed point problem, which allows to calibrate precisely the connection kernel of the larger graph. This is done, under sufficient conditions, in the following proposition.

\begin{proposition}
 \label{lemma: invertibility}
    Let Assumption~\ref{ass:main} hold.
    The system of equations
    \[
    \sum_{s=1}^d(w_{rs}+w_{sr})\frac{2y_{r}y_{s}}{y_{r}+y_{s}} =  \nu_r\log\left(\frac{1+y_{r}}{1-y_{r}}\right)
    \]
    admits a unique solution $\B y$ that lies in $(0,1)^d$ for any 
    $\B w$ such that $\B w\B 1\succ \B \nu$. Moreover, for any sequence $(\B \nu_n,\B a_n)\to (\B \nu,\B a)$, all satisfying Assumption~\ref{ass:main}, then $\B y_n\to \B y$.
\end{proposition}
\begin{proof} 
The proof is divided into four steps.

\emph{{\bf Step 1.} There is a positive vector $\B x$ such that $\sum_s (w_{rs}+w_{sr})\frac{x_s}{x_r+x_s}=\sum_s w_{rs}$ for all $r\in \mc{S}$.}

We borrow the proof of this from the analysis of the well-known Bradley–Terry model (see \textit{e.g.} \cite{zermelo1929} or \cite{ford_solution_1957}).
Let $S=\{\B \gamma\in (0,\infty)^d:\sum_{r=1}^d \gamma_{r}=1\}$ and consider the function $F_{\B w}: S \to(0,1]$,
\[
    F_{\B w}(\B \gamma) = \prod_{r< s} \left(\frac{\gamma_r}{\gamma_r+\gamma_s}\right)^{w_{rs}} \left(\frac{\gamma_s}{\gamma_r+\gamma_s}\right)^{w_{sr}}.
\] 
We prove that $F_{\B w}$ attains a maximum in $S$ at $\B \gamma^*\in S$. First, by standard optimization, it follows that, if a maximizer $\B \gamma^*$ exists, then it must satisfy 
\[
    \sum_{s=1}^d w_{rs} - \sum_{s=1}^d (w_{rs}+w_{sr})\frac{\gamma^*_r} {\gamma^*_r+\gamma^*_s} = \lambda \gamma^*_r,
\]
for all $r\in \mc{S}$ and for a multiplier $\lambda\in \R$. By summing over $r$ we see that $\lambda=0$ necessarily; the claim follows by taking $x_r = 1/\gamma^*_r$. It only remains to prove that such $\B \gamma^*\in S$ exists. We take $\overline{S}$ -- the closure of $S$ -- and show that $F_{\B w}$ can be extended continuously to $0$ at the boundary $\partial S$. Fix $\B\gamma ^\ssup{0} \in \partial S$. There must be indices $r\ne s$ such that $\gamma^\ssup{0}_r=0$ and $\gamma^\ssup{0}_s>0$. By the irreducibility of $\B w$ there is a sequence of mutually distinct indices $\smash{r=i_0,i_1,\dots, i_{m-1}, i_m=s}$ such that $\smash{w_{i_ji_{j+1}}>0}$ for any $j=0,\dots, m$, therefore there must be indices $k,l$ such that $\gamma^\ssup{0}_k=0$, $\gamma^\ssup{0}_l>0$, and $w_{kl}>0$. Now let $\B \gamma\in S$ and extract the $kl$ term: we can write $F_{\B w} (\B \gamma) = \left(\frac{\gamma_k}{\gamma_k+ \gamma_l} \right)^{w_{kl}} \phi_{\B w}(\B \gamma)$,
with $\phi_{\B w}(\B \gamma)\in (0,1]$. It follows $\lim_{\B \gamma \to \B \gamma^\ssup{0}}F_{\B w} (\B \gamma)=0$, hence $\B \gamma^*$ exists and it is necessarily in $S$, not at the boundary.

\textbf{Step 2.}  \emph{ Consider the vector field 
    \[
    V_r(\B y) = \sum_{s=1}^d (w_{rs}+w_{sr}) \frac{2y_{r}y_{s}}{y_{r}+y_{s}} - \nu_r\log\left(\frac{1+y_{r}}{1-y_{r}}\right) \qquad {r=1,2,\dots, d}.
    \]
    Then there is $\B y \in (0,1)^d$ such that $\B V(\B y)= \B0$.}

     We use the vector $\B x$ from \emph{Step 1} to slighlty adapt the the proof of Poincaré–Miranda theorem for the existence of a fixed point. Fix a small $\eps>0$, the positive vector $\B x$ from \emph{Step 1} and consider the hypercube $H= [x_1\eps, 1-\eps]\times [x_2\eps, 1-\eps] \times \dots \times [x_d\eps, 1-\eps]$. Define for all $r$ the opposite pairs of the hypercube:
    \[
    H^+_r = \{\B y\in H: y_r=x_r\eps\} \qquad H^-_r = \{\B y\in H: y_r=1-\eps\} .
    \]
    By the Poincaré–Miranda theorem, if $V_r(H^+_r)>0$ and $V_r(H^-_r)<0$ for all $r$, then there exists a point $\B y^*$ in the interior of the hypercube such that $\B V(\B y^*) = \B 0$. Fix $r\in \mc{S}$.
  First, let $\B y \in H^-_r$, then we have $V_r(\B y) \le 2\sum_{s=1}^d(w_{rs}+w_{sr}) - \nu_r\log\left(\frac{2-\eps}{\eps}\right) < 0$ for sufficiently small $\eps$.
              Now, let $\B y \in H_r^+$. The minimum of $V_r$ on $H^+_r$ is achieved at $\B y = \eps \B x$, therefore, by the previous step,
            \begin{align*}
            {V_r(\eps\B x)} &\ge x_r\eps\left(\sum_{s=1}^d (w_{rs}+w_{sr})\frac{x_s}{x_r+x_s} - \frac{\nu_r}{2x_r\eps} \log\left(\frac{1+x_r\eps}{1-x_r\eps}\right)\right) =x_r\eps\left(\sum_{s=1}^d w_{rs} - {\nu_r}(1+o(\eps))\right)
            \end{align*}
            and the last term is positive for sufficiently small $\eps$. The thesis follows from Poincaré–Miranda theorem.

 \emph{\bf Step 3.} \emph{Uniqueness in $(0,1)^d$.} 
 Let $\B y, \B z\in(0,1)^d$ be solutions and suppose $t=\max_s y_s/z_s>1$ is attained at $r$. Notice that
\[
\nu_r\log\left(\frac{1+y_{r}}{1-y_{r}}\right) = \sum_s (w_{rs}+w_{sr})\frac{2y_ry_s}{y_r+y_s}
 \le t\sum_s (w_{rs}+w_{sr})\frac{2z_rz_s}{z_r+z_s}=t\nu_r\log\left(\frac{1+z_{r}}{1-z_{r}}\right) ,
\]
since $\frac{y_ry_s}{y_r+y_s}$ is increasing in both variables.
However, the function $g(u)=\log((1+u)/(1-u))$ is strictly convex on $(0,1)$ with $g(0)=0$, so $g(tz_r)>t g(z_r)$. This contradicts the inequality above.  Thus it has to be true that $\B y\preceq \B z$. Repeating the argument with opposite sign gives $\B z\preceq \B y$. Hence $\B y=\B z$.

   \emph{\bf Step 4.} \emph{Continuity of $\B y(\B \nu, \B a)$.} Let $b_{rs}=w_{rs}+w_{sr}$ and $\mathcal{J}=D_y\B V(\B y)$. Then
\[
\mathcal{J}_{rs}=
\begin{cases}
\displaystyle \frac{2b_{rs}y_r^2}{(y_r+y_s)^2}, & r\neq s,\\[1ex]
\displaystyle b_{rr}+\sum_{k\neq r}\frac{2b_{rk}y_k^2}{(y_r+y_k)^2}
-\frac{2\nu_r}{1-y_r^2}, & r=s.
\end{cases}
\]
At a solution $\B V(\B y)=0$, we have that 
\[
(\mathcal{J}y)_r=
\nu_r\left(\log\frac{1+y_r}{1-y_r}
-\frac{2y_r}{1-y_r^2}\right).
\]

Now, see that $\log\frac{1+u}{1-u}=\int_{0}^u \frac{2}{1-t^2}dt<\int_{0}^u \frac{2}{1-u^2}dt=\frac{2u}{1-u^2}$ and therefore $(\mathcal{J}y)_r<0$. Notice that $-(\mathcal{J}y)_r$ is the sum of the $r$-th row of the matrix $-\mathcal{J}\operatorname{diag}(y)$.
Now, see that
\[
-\mathcal{J}_{rr}y_r> \sum_{s\neq r}\mathcal{J}_{rs}y_s\geq 0,
\]
we see that the matrix $-\mathcal{J}\operatorname{diag}(y)$ has positive diagonal and is strictly
row diagonally dominant. Hence $-\mathcal{J}\operatorname{diag}(y)$ is invertible, but since $y_r>0$ for any $r$, this implies that $\mathcal{J}=D_y\B V(\B y)$ is invertible in the unique solution $\B y$. This implies that the solution $\B y$ is continuous with respect to the parameters $(\B \nu,\B a)$.
\end{proof}

\subsection{Proof of the theorem}
Now we are ready to prove the main theorem. We exploit the fact that, since the edge matrix concentrates exponentially fast around the optimal one in a connected inhomogeneous random graph, it is enough to tune the connection kernel to make any precise the limiting connection matrix the typical one for that graph. Hence, we choose a limit kernel $\B \kappa$ such that $\B a$ is its optimal (rescaled) edge matrix.

\begin{proof}[Proof of Theorem~\ref{thm:main}]
 By Proposition~\ref{lemma: invertibility} the set of equations \eqref{eq: y_rs_equation} admits a unique solution $\B y\in (0,1)^d$. We choose the following $\B \kappa$: 
\[
\kappa_{rs}= \frac{2^{\delta_{rs}}a_{rs}}{\nu_r\nu_s}\frac{2y_ry_s}{y_r+y_s} \qquad r,s \in \mc{S},
\]
which is symmetric and irreducible by construction. It it is easy to see that the set of equations \eqref{eq: y_rs_equation} is equivalent to $
    y_r^{-1} = \coth(\frac{1}{\nu_r}\sum_{s}2^{\delta_{rs}}a_{rs}\frac{y_ry_s}{y_r+y_s}) =\coth(\frac 12(\B \kappa \B \nu)_r),
$
which implies 
\[
    2^{\delta_{rs}}a_{rs}  = \frac{\nu_r\kappa_{rs}\nu_s}{2} \left(\coth(\frac{1}{2} ( \B\kappa \B \nu )_r)+\coth(\frac{1}{2} (\B \kappa \B \nu )_s) \right).
\]
Now, consider the inhomogeneous random graph $\Gcal(\B n, \B \kappa)$ with connection kernel $\B \kappa$ as above. From Theorem~\ref{th: limit_connectivity} and the choice of $\B \kappa$, the normalized logarithm of the probability of connectedness has limit 
\[
q(\B \nu, \B \kappa)= \sum_{r} \nu_r \log\frac{2y_r}{1+y_r}.
\]
On the other hand, by Proposition~\ref{th: typical_size_multid}, $q(\B \nu, \B \kappa) = \lim_{n\to\infty} \frac 1n \log q_n (\B n, \B \kappa^\ssup{n}, \B \eps_n)$ for some suitably chosen $\B \kappa^\ssup{n}\to \B \kappa$, therefore
\[
\begin{aligned}
\sum_{r} \nu_r \log\frac{2y_r}{1+y_r} &= \lim_{n\to \infty} \ld{ \left[ C(\B n, \B \eps_n) P_{\B n, \B \kappa^\ssup{n}} (\B \eps_n)\right]} \\
&= \Phi(\B n, \B \eps_n) + \lim_{n\to \infty} \ld{ \left[ B(\B n, \B \eps_n) P_{\B n, \B \kappa^\ssup{n}} (\B \eps_n)\right]},
\end{aligned}
\]
where we wrote $C(\B n, \B \eps_n) = e^{n\Phi{(\B n, \B \eps_n)}} B(\B n, \B \eps_n)$. Consider the asymptotics of each factor in $P_{\B n, \B \kappa^\ssup{n}} (\B \eps_n)$. For all $r,s\in \Scal$:
\begin{align*}
        \frac 1n \log \left[\left(\frac{\kappa^\ssup{n}_{rs}/n}{1-\kappa^\ssup{n}_{rs}/n}\right)^{\eps_{rs}}(1-\kappa^\ssup{n}_{rs}/n)^{J_{rs}}\right] &=  a_{rs}(\log\kappa_{rs}-\log n) - \frac{\nu_r\kappa_{rs}\nu_s}{2^{\delta_{rs}}}+ o(1) \\
        & =  a_{rs}\left(\log\kappa_{rs}- \frac{2y_ry_s}{y_r+y_s} -\log n\right)+ o(1).
    \end{align*}
On the other hand, expanding  the binomial terms in $B(\B n, \B \eps_n)$ with Stirling's formula, we get that, for all $r,s\in \mc{S}$:
\begin{align*}
    \ld{\binom{J_{rs}}{\eps_{rs}}} &= \frac{\eps_{rs}}{n} \log{\left(\frac{J_{rs}}{\eps_{rs}}\right)} + \frac{\eps_{rs} - J_{rs}}{n}\log \left(1-\frac{\eps_{rs}}{J_{rs}}\right) + o(1)\\
    & = a_{rs}\left(\log\frac{\nu_r\nu_s}{2^{\delta_{rs}}a_{rs}}+\log n\right) - \frac{J_{rs}}{n}\log \left(1-\frac{\eps_{rs}}{J_{rs}}\right) + o(1) \\
    &= a_{rs}\left(\log\frac{2y_ry_s}{\kappa_{rs}(y_r+y_s)}+\log n + 1 \right) +o(1).
\end{align*}
It follows that $\Phi(\B n, \B \eps_n) = \varphi(\frac{\B n}{n}, \frac{\B \eps_n}{n}) + o(1)$, with
\[
\varphi(\B \nu, \B a) =  \sum_{r} \nu_r \log\frac{2y_r}{1+y_r} -\sum_{r\le s} a_{rs}\left[1+\log\left(\frac{2y_ry_s}{y_r+y_s}\right) -\frac{2y_ry_s}{y_r+y_s}\right].
\]
The thesis follows by continuity of $\varphi(\B \nu, \B a) $ in both variables.
\end{proof}
\appendix
\section{Appendix} 

\subsection{Connectivity constraints for multitype edge matrices}\label{sec: trees}
Throughout this appendix, $\B n=(n_r)_{r\in\mc S}$ is fixed with $n_r>0$ and $n=\sum_r n_r$. We call an edge matrix $\B\varepsilon$ \emph{feasible} if $C(\B n,\B\varepsilon)>0$. Clearly, we have the capacity constraints, that is $\B\varepsilon\preceq\B J$, where $\B J$ is defined in~\eqref{eq:edges_max}. Moreover, feasibility imposes genuinely multitype connectivity constraints. A useful way to separate these constraints from the extra edges is to reduce to the existence of spanning trees. Indeed $C(\B n,\B\varepsilon)>0$ if and only if $\B 0\preceq\B\varepsilon\preceq\B J$ and there exists an edge matrix $\B\tau\preceq\B\varepsilon$ that is realized by a spanning tree on $V_{\B n}$.

It is clear that the main question is whether there exists or not a spanning tree with prescribed type-pair counts. A characterization can be expressed through directed trees. Fix a root type $\rho\in\mc S$ and orient every edge of a spanning tree toward a root vertex of type $\rho$. Let $m^{(\rho)}_{rs}$ denote the number of directed edges whose tail has type $r$ and whose head has type $s$. Then
\begin{equation}\label{eq:directed-tree-conditions-revised}
    \tau_{rs}=\frac{m^{(\rho)}_{rs}+m^{(\rho)}_{sr}}{2^{\delta_{rs}}}.
\end{equation}
If $\mc L(\B m^{(\rho)})=\operatorname{diag}(\B m^{(\rho)}\B 1)-\B m^{(\rho)}$ is the directed Laplacian, then necessary conditions for the existence of such $\B m ^{(\rho)}$ are
\begin{equation}\label{eq:directed-cofactor-revised}
    \sum_{s\in\mc S} m^{(\rho)}_{rs}= n_r-\mathds 1_{\{r=\rho\}}, \qquad \det\big(\mc L_{-\rho}(\B m^{(\rho)})\big)\ne 0,
\end{equation}
 where $\mc L_{-\rho}$ is obtained from $\mc L$ by deleting the row and column indexed by $\rho$. The directed matrix-tree condition -- the second equation in ~\eqref{eq:directed-cofactor-revised} -- is equivalent to the weighted digraph on $[d]$ vertices encoded by $\B m^{(\rho)}$ containing a spanning tree rooted at $\rho$. The multitype tree enumeration framework in~\cite{BeMo14} establishes these conditions are also sufficient. In this form, the characterization says that a tree edge matrix $\B\tau$ is feasible if and only if, for some root type $\rho$, there is a nonnegative integer matrix $\B m^{(\rho)}$ satisfying~\eqref{eq:directed-tree-conditions-revised} and~\eqref{eq:directed-cofactor-revised}. 

Several simple necessary conditions follow immediately. Equations ~\eqref{eq:directed-cofactor-revised} imply, respectively, 
\begin{equation} \label{eq: tree-necessary-revised-1}
    \sum_{r\leq s}\tau_{rs}=n-1 \quad \text{ and } \quad \B \tau \text{ irreducible}.
\end{equation}
Moreover, if $\B\tau$ is realized by a spanning tree, then
\begin{equation} \label{eq:tree-necessary-revised-2}
    \sum_{\substack{r\leq s\\ r,s\in A}}\tau_{rs}\leq \sum_{r\in A} n_r-1
\quad\text{for every nonempty }A\subsetneq\mc S.
\end{equation}
Indeed, the edges with both endpoints in the set of types $A$ form a forest on $\sum_{r\in A}n_r$ vertices. In particular, taking $A=\{r\}$ gives $\tau_{rr}\leq n_r-1$.
In general, for a connected graph with excess $k=\sum_{r\le s}\varepsilon_{rs}-(n-1)>0$, the corresponding necessary bounds are \begin{equation}\label{eq:connected-necessary-revised}\sum_{\substack{r\leq s\\ r,s\in A}}\varepsilon_{rs}\leq \sum_{r\in A} n_r-1+k,\qquad \emptyset\ne A\subsetneq \mc S.\end{equation}
These subset inequalities are stronger than the one-type inequalities $\varepsilon_{rr}\leq n_r-1+k$ and are often the first obstruction in dimensions $d\geq3$.

\paragraph{The two-type case.}
When $d=2$, the preceding conditions reduce to a clean criterion. Let $\B n=(n_1,n_2)$ and let $k=\varepsilon_{11}+\varepsilon_{12}+\varepsilon_{22}-(n-1)$.
Then $C(\B n,\B\varepsilon)>0$ if and only if $k\geq0$, the capacity constraints hold, and
\begin{equation}\label{eq:two-type-feasibility-revised}
\varepsilon_{12}\geq 1,
\qquad
\varepsilon_{ii}\leq n_i-1+k\quad (i=1,2).
\end{equation}
Necessity follows from irreducibility and~\eqref{eq:connected-necessary-revised}. For sufficiency, choose integers $\tau_{ii}\leq \min\{\varepsilon_{ii},n_i-1\}$ so that
\[
\tau_{12}:=n-1-\tau_{11}-\tau_{22}
\]
satisfies $1\leq \tau_{12}\leq\varepsilon_{12}$; the inequalities in~\eqref{eq:two-type-feasibility-revised} guarantee that this can be done. Build forests on the two type classes with respectively $\tau_{11}$ and $\tau_{22}$ internal edges. These forests have $n_1-\tau_{11}$ and $n_2-\tau_{22}$ components. Since
\[
\tau_{12}=(n_1-\tau_{11})+(n_2-\tau_{22})-1,
\]
one may connect these components by a bipartite tree using exactly $\tau_{12}$ cross-type edges. This gives a spanning tree with edge matrix $\B\tau\preceq\B\varepsilon$, yielding a connected graph with edge matrix $\B\varepsilon$.

\paragraph{Relation with Assumption~\ref{ass:main}.}
The decomposition in Assumption~\ref{ass:main}, used to prove Theorem~\ref{thm:main}, has a natural interpretation. Given a connected graph, choose a spanning tree and orient it toward a root; orient all remaining edges arbitrarily. If $\omega_{rs}$ is the number of oriented edges from type $r$ to type $s$, then
\[
2^{\delta_{rs}}\varepsilon_{rs}=\omega_{rs} + \omega_{sr},
\qquad
\B \omega \B 1\succeq \B n-\B e_{\rho},
\]
where $\B e_{\rho}$ is the coordinate vector of the root type. Then, rescaling the matrices $\B \eps$ and $\B \omega$ with $n$, we see that any limiting connected edge matrix $\B a$ therefore admits a decomposition with $\B\nu\preceq \B w\B 1$. With the idea to have a non-negligible excess of edges (equivalent to the easy assumption $k\approx xn$ for $x>0$ in the one-type case), we impose 
the strict condition $\B\nu\prec\B w\B 1$ in Assumption~\ref{ass:main}.
The irreducibility of $\B w$ rules out a decomposition supported on disconnected type classes and prevents the fixed-point solution in~\eqref{eq: y_rs_equation} from approaching the boundary.

\section*{Acknowledgements}
LA would like to acknowledge fruitful discussions with B. Chin, M. Dickson, S. Donderwinkel, R. van der Hofstad and N. Malhotra
on connected random graphs initiated during \emph{RandNET Summer School and Workshop on Random Graphs} in Eindhoven in August 2022, and with G. Bet and M. Phalempin during her stay at the University of Florence. LA acknowledges the partial support by the
INDAM GNAMPA Project, codice CUP E53C25002010001 and by the Italian Ministry of University and Research
(MUR) via PRIN 2022 – ConStRAINeD-CUP-2022XRWY7W.  
MV acknowledges financial support from the DFG through the IRTG 2544 “Stochastic Analysis in Interaction.

The authors acknowledge the use of generative AI tools to assist with literature searching, typesetting, and manuscript editing of the latest version.  The principal ideas, mathematical arguments, and proof strategies were developed and independently verified by the authors, that take full responsibility for the final proofs, accuracy, and integrity of the work.

\bibliographystyle{abbrv}
\bibliography{article}  
\end{document}